\documentclass[11pt, a4paper reqno]{amsart}

\usepackage[dvipsnames]{xcolor}
\definecolor{darkblue}{rgb}{0,0,0.7}
\definecolor{darkred}{rgb}{0.7,0,0}
\usepackage[colorlinks=true, linkcolor=darkred, citecolor=darkblue, urlcolor=blue, pagebackref=true, breaklinks=true]{hyperref}
\usepackage[capitalise]{cleveref}
\usepackage{wrapfig}

\usepackage{amscd}
\usepackage{amssymb,amsmath}

\usepackage{mathrsfs}
\usepackage{graphicx}
\usepackage{enumitem}
\usepackage{mathtools}
\usepackage[utf8]{inputenc}
\usepackage{listings}
\usepackage{lipsum,eso-pic,xcolor}
\usepackage{lineno}
\usepackage{tikz}
\usepackage{caption}
\usetikzlibrary{arrows,decorations.pathmorphing,backgrounds,positioning,fit}
\usetikzlibrary{positioning}

\usetikzlibrary{trees}
\usetikzlibrary{arrows}

\usepackage{placeins}
\usepackage{relsize}
\usepackage{todonotes}
\usepackage{soul}
\usepackage{enumitem}
\usepackage{setspace}
\usepackage{wrapfig}
\usepackage{etoolbox}
\BeforeBeginEnvironment{wrapfigure}{\setlength{\intextsep}{0pt}}
\newtheorem{proposition}{Proposition}[section]
\newtheorem{lemma}[proposition]{Lemma}
\newtheorem{theorem}[proposition]{Theorem}

\newtheorem{construction}[proposition]{Construction}

\theoremstyle{definition}
\newtheorem{remark}[proposition]{Remark}

\newtheorem{example}[proposition]{Example}
\newtheorem{definition}[proposition]{Definition}
\newtheorem{notation}[proposition]{Notation}

\newcommand{\reg}{{\rm reg}}

\DeclareMathOperator{\fcdeg}{fcd}

\DeclareMathOperator{\hs}{HS}
\DeclareMathOperator{\hf}{HF}
\DeclareMathOperator{\hp}{HP}
\DeclareMathOperator{\cdeg}{cd}

\def\K{\mathbb{K}}

\def\reg{\mathrm{reg}}

\def\height{\mathrm{ht}}

\DeclareMathOperator{\Tor}{Tor}

\begin{document}
	
	\title[Hilbert Coefficients and Regularity of Binomial Edge Ideals]{Hilbert Coefficients  and Regularity of Binomial Edge Ideals}

	\author{Kanoy Kumar Das}
	\address{Chennai Mathematical Institute, India}
	\email{kanoydas0296@gmail.com; kanoydas@cmi.ac.in}

    \author[Rajiv Kumar]{Rajiv Kumar}
    \email{rajiv.kumar@iitjammu.ac.in}
    \address{Department of Mathematics, Indian Institute of Technology Jammu, J\&K, India - 181221.}

    \author[Paramhans Kushwaha]{Paramhans Kushwaha}
    \email{2022rma2004@iitjammu.ac.in}
    \address{Department of Mathematics, Indian Institute of Technology Jammu, J\&K, India - 181221.}

	
	\vspace*{-0.4cm}
	\begin{abstract}
		Let $G$ be a simple graph on $n$ vertices, and let $J_G$ denotes the corresponding binomial edge ideal in $S=\mathbb{K}[x_1,\ldots,x_n,y_1,\ldots,y_n]$, where $\K$ is a field. We show that if a vertex satisfies a certain degree condition, then some Hilbert coefficients remain unchanged upon its removal, thereby providing a reduction technique for computing Hilbert coefficients. As an application, for any $i\geq 0$ and a pair $(r,s)$ with $r\geq 2, s\in \mathbb Z$, we show that there always exists a graph $G$ such that $\mathrm{reg}(S/J_G)=r$ and $e_i(S/J_G)=s$, where $\mathrm{reg}(S/J_G)\text{ and } e_i(R/J_G)$ denote the Castelnuovo–Mumford regularity and the $i$-th Hilbert coefficient of $S/J_G$, respectively. In particular, this demonstrates that there is no inherent relationship between the regularity and the Hilbert coefficients for the class of binomial edge ideals. 
	\end{abstract}
    
	\maketitle
	
	\section{Introduction}
	Let $G$ be a finite simple graph with vertex set $[n]:=\{1,2, \ldots,n\}$ and edge set $E(G)$. Consider $S=\K[x_1,\ldots,x_n,y_1,\ldots, y_n]$ a polynomial ring over a field $\K$. Herzog et al. \cite{HHHFT2010-BEI-and-conditional-independence}, and independently Ohtani \cite{GraphsandIdealsGeneratedbysome2-minor-Ohtani-Masahiro} introduced a notion of binomial edge ideal corresponding to a finite simple graph. The \emph{binomial edge ideal} of $G$, denoted $J_G$, is defined as 
    \[
    J_G:=\left \langle x_iy_j-x_jy_i \mid ij \in E(G) \right \rangle \subseteq S.
    \]
    Combinatorially, $J_G$ can be interpreted as the ideal generated by a subset of maximal minors of a generic $2\times n$ matrix, corresponding precisely to the edges of $G$. This framework simultaneously generalizes two constructions: the determinantal ideal and the ideal of adjacent minors of a $ 2 \times n$ generic matrix. 

    The Hilbert polynomial of a finitely generated graded algebra carries fundamental structural information and can be computed from its minimal free resolution. The coefficients of the Hilbert polynomial, with certain binomial adjustments, are known as the Hilbert coefficients, among which the leading coefficient, often denoted by $e_0$, is known as the multiplicity. Multiplicity, an invariant with rich geometric interpretation, has been studied intensively for decades. Nevertheless, its computation for sufficiently general ideals is highly nontrivial, and the higher Hilbert coefficients are even less tractable. Although binomial edge ideals have been extensively studied, their Hilbert series and Hilbert functions have been addressed only in a limited number of works, namely \cite{KS2019-Hilbert-series-of-BEI, MohammadiSharifan2014, SchenzelZafar2014, ZafarZahid2013}. In particular, the multiplicity of binomial edge ideals has not been computed previously, and no general formula for it appears in the existing literature.

    In this work, we establish a general necessary condition on a vertex under which some Hilbert coefficients remain unchanged after its removal. More precisely, to each vertex $v$ of a graph $G$, we assign a graph-theoretic invariant, called the free-clique degree of $v$ (denoted $\fcdeg_G(v)$; see \Cref{def: FCD}), and prove the following.

    \medskip
    \noindent
    \textbf{Theorem~~\ref{thm: delete vertex}.} 
    Let $G$ be a graph and $v\in V(G)$ be such that $\fcdeg_G(v)\geq i+3$ for some  $i\geq 0$. ~Then 
		\[
		e_i(S/J_G)=e_i(S/J_{G\setminus{v}}).
		\]

    In the area of combinatorial commutative algebra, a popular line of study is the identification or existence of a class of ideals with a pair, or a tuple of algebraic invariants. Such studies have been carried out for the class of edge ideals of graphs in \cite{HibiKannoKimuraMatsudaVanTuyl2021, HibiKimuraMatsudaTsuchiya2021,  HibiMatsudaVanTuyl2019}, and the references therein; for binomial edge ideals in \cite{HibiMatsuda2022}. In this work, we consider the class of binomial edge ideals and study the pairs $(r,s)$, where $r=\reg(R/J_G)$ and $s=e_i(R/J_G)$ is the $i$-th Hilbert Coefficient. As an application of \Cref{thm: delete vertex}, we show that within the class of binomial edge ideals, any such pair can be realized. In particular, we prove the following.

    \medskip
    \noindent
    \textbf{Theorem~~\ref{thm: main}.} 
        Given $i\geq 0$ and a pair $(r,s)$ with $r\geq 2, s\in \mathbb Z$, there is a graph $G$ such that 
        \[
        \reg(S/J_G)=r \text{ and } e_i(S/J_G)=s.
        \]
    It is worth noting that in the case $\reg(S/J_G)=r=1$, that is, when $J_G$ has linear resolution, all the pairs of the form $(1,s)$, where $s=e_i(R/J_G)$ for some $i\geq 0$, are completely determined. This follows from the fact that $J_G$ has linear resolution if and only if $G$ is a complete graph (\cite[Theorem 2.1]{MadaniKiani2012}). In this case, $J_G$ is the determinantal ideal, and therefore its Hilbert polynomial is completely known (see \cite{ConcaHerzog1994}).

    This article is organized as follows. In \Cref{sec: prelim}, we recall the necessary algebraic and graph-theoretic concepts related to binomial edge ideals. \Cref{sec: main} contains the main result of the article, where we prove that certain Hilbert coefficient remains unchanged after the removal of a vertex with certain free-clique degree. This provides an inductive tool for determining the higher Hilbert coefficients of binomial edge ideals. An application of this result is given in \Cref{sec: application}. In \Cref{thm: reg=2 is sufficient for higher reg}, we present a method for constructing graphs with arbitrarily large regularity and prescribed Hilbert coefficients, starting from graphs of small regularity. Finally, in \Cref{thm: main} we show that for any $i\geq 0$ and integers $r\geq 2, s$, any pair $(r,s)$, where $r=\reg(S/J_G)$ and $s=e_i(S/J_G)$ can be realized.

	\section{Preliminaries}\label{sec: prelim}
	In this section, we recall all preliminary notions and results related to commutative algebra and combinatorics. Let $R=\K[z_1,\ldots , z_n]$ be a polynomial ring over a field $\K$ with $\deg(z_i)=1$ for all $1\leq i\leq n$, and $I\subseteq R$ a homogeneous ideal. The Castelnuovo-Mumford regularity(or simply, regularity) is a homological invariant that measures the complexity of the minimal free resolution of $R/I$. This invariant can be defined in terms of non-vanishing of the multigraded $\mathrm{Tor}$-modules of $R/I$ with respect to the field $\K$:
	\[
	\reg(R/I):=\sup\{j-i \mid \dim_{\K}(\Tor_i^R(R/I,\K)_j)\neq 0\}.
	\]
	This invariant has been extensively studied for several classes of ideals, including various families of monomial ideals and binomial edge ideals (see, for instance \cite{Banerjee2019, Ha2014}, and the recent survey \cite{jayanthan2025generalizedbinomialedgeideals}).

    We now review the fundamental concepts of Hilbert functions and Hilbert coefficients. In this article, we follow the convention that the set of all natural numbers $\mathbb{N}$ contains $0$.  The \emph{Hilbert function} of the graded ring $R/I$ is the function $\hf(R/I, \rule{0.2cm}{0.15mm}):\mathbb{N}\longrightarrow \mathbb{N}$ with 
    \[
    \hf(R/I, j)=\dim_{\K}(R/I)_j
    \]
    for all $j\in \mathbb N$, where $(R/I)_j$ denotes the $j$-th graded component of $R/I$. The generating function of this numerical function is called the \emph{Hilbert Series} of $R/I$, and is given as follows
    \[
    \hs(R/I,t)=\sum_{j\in \mathbb{N}}\hf(R/I,j)\ t^j.
    \]
    A classical theorem of Hilbert-Serre provides a reduced form of the Hilbert series
    \[
    \hs(R/I,t)=\frac{Q(t)}{(1-t)^d},
    \]
    where $Q(t)\in \mathbb{Z}[t]$ is a polynomial with $Q(1)\geq 1$ and $d=\dim(R/I)$. As a consequence, the Hilbert function, $\hf(R/I,j)$ eventually agrees with a polynomial in $j$ of degree $d-1$.
    This polynomial is called the Hilbert polynomial of $R/I$. Moreover, the coefficients of the Hilbert polynomial can be recovered from the numerator $Q(t)$. More precisely, if $\hp(R/I, X)$ denotes the Hilbert polynomial of $R/I$, and is written in the form
    \[
    \hp(R/I, X)= \sum_{i=1}^{d-1}(-1)^{d-1-i}e_{d-1-i} \left( \begin{array}{c} X+i \\ i \end{array} \right),
    \]
    then the coefficients $e_i \text{ for }0\leq i\leq d-1$ can be obtained as  follows: $e_i=\frac{Q^{(i)}(1)}{i!}$, where $Q^{(i)}(t)$ denotes the $i$-th derivative of the polynomial $Q(t)$. The number $e_i$ is called the 
    \emph{$i$-th Hilbert coefficient} of $R/I$, and   be denoted by $e_i(R/I)$ for $0\leq i\leq d-1$. In particular, when $\dim(R/I)>0$, the coefficient $e_0(R/I)$ is the well-known \emph{multiplicity} of $R/I$. 
	
	A graph consists of a pair $(V(G), E(G))$, where $V(G)$ is called the set of vertices, and $E(G)\subseteq 2^{V(G)}$ is the set of edges of $G$. In this article, we consider only simple graphs, that is, graphs with no loops or multiple edges.  Throughout this article, unless explicitly stated otherwise, we take $V(G)=[n]=\{1,\dots , n\}$, and we   write $E(G)=\{ab\mid \text{there is an edge between }a \text{ and } b\}$.  For simplicity, we sometimes specify only the edge set of a graph and write 
    $G=E(G)$, without explicitly mentioning the vertex set. If $G$ is not connected, and if $G_1,\ldots , G_r$ be the connected components of $G$, we   write $G=G_1\sqcup \cdots \sqcup G_r$. In the case when $E(G)=\emptyset$, we simply write $G=\{a_1,\dots , a_r\}$ to indicate that the graph $G$ consists only of vertices $a_1,\dots , a_r$ of $G$. A \emph{subgraph} $H$ of $G$ is a graph where $V(H)\subseteq V(G) \text{ and } E(H)\subseteq E(G)$. An \emph{induced subgraph} of $G$ on a vertex set $T\subseteq V(G)$ is the subgraph $H$ with $V(H)=T$ and $E(H)=\{ab\mid ab\in E(G)\text{ and } a,b\in T\}$. Given a subset $W\subseteq V(G)$, $G\setminus W$ is the induced subgraph of $G$ on the vertex set $V(G)\setminus W$. For $W=\{v\}$, we simply write $G\setminus v$. A \emph{complete graph} or a \emph{clique} $K_r, r\geq 1$   is a graph on $[r]$ such that $ij\in E(G)$ for all $i\neq j\in [r]$. We say that a subgraph $H$ of $G$ is a \emph{clique of $G$} if $H$ is a complete graph. A maximal clique of $G$ is a clique that is not an induced subgraph of any other clique of $G$. For $v\in V(G)$, let $\cdeg_G(v)$ denote the number of the maximal cliques that contain $v$, called the \emph{clique degree} of $v$. A vertex of clique degree one is called a \emph{free} (or, \emph{simplicial}) vertex. For example, any vertex of a complete graph is a free vertex. A \emph{complete bipartite graph}, denoted by $K_{n,r}, n,r\geq 1$, is a graph with  $V(K_{n,r})=[n+r]$ and $E(K_{n,r})=\{ij\mid 1\leq i\leq n, n+1\leq j\leq r\}$. For any graph-theoretic terminology not explained here, we refer the reader to look at the standard text of graph theory \cite{BondyMurty2008}.

	We now recall several results related to the main algebraic object of this article, the binomial edge ideal. This class of ideals have been studied extensively in the literature, and here, we recall the properties that we intend to use later in this article. Let $G$ be a graph on the vertex set $[n]$ and $J_G\subseteq S=\K[x_1,\ldots, x_n, y_1, \ldots , y_n]$ be the binomial edge ideal of $G$. Throughout the remainder of this article, $S$   denote the polynomial ring with the appropriate number of variables in which the ideal $J_G$ is defined. The associated primes of $J_G$ are first determined in \cite{HHHFT2010-BEI-and-conditional-independence}, and they are given using the connectivity property of the graph $G$. To understand the associated primes and the primary decomposition of $J_G$, we require a few notations. Let $T\subseteq [n]$ be a subset of vertices of $G$, and consider $G\setminus T$, the induced subgraph of $G$ on the vertex set $[n]\setminus T$. Suppose that $G_1,\ldots , G_{c_G(T)}$ be the connected components of $G\setminus T$. For $1\leq i\leq c_G(T)$, let $\widetilde{G_i}$ denote the complete graph on the vertex set $V(G_i)$. Now consider the ideal 
    \[
    P_T(G)\coloneqq \left \langle \bigcup_{i\in T}\{x_i,y_i\} , J_{\widetilde{G_1}}, \ldots ,  J_{\widetilde{G_{c(T)}}}  \right\rangle.
    \]
    The ideal $P_T(G)$ is a prime ideal of $S$, and $\height(P_T(G))=n+|T|-c_G(T)$. Moreover, the primary decomposition of $J_G$ (see \cite[Theorem 3.2]{HHHFT2010-BEI-and-conditional-independence}) is given by 
    \[
    J_G=\bigcap_{T\subseteq [n]}P_T(G).
    \]
    Thus, we have $\dim(S/J_G)=\max\{n-|T|+c_G(T):T\subset [n]\}$. It should be noted that the above primary decomposition of $J_G$ is not always minimal. To identify the minimal prime ideals among $P_T(G)$, we require a graph-theoretic notion known as the \emph{cut point property}. A vertex $v\in V(G)$ is called a \emph{cut point} of $G$ if the removal of the vertex $v$ increases the number of connected components of $G$. If each vertex $i\in T$ is a cut point of the induced subgraph $G\setminus (T\setminus \{i\})$ of $G$, i.e., $c_{G\setminus T}> c_{G\setminus (T\setminus \{i\})}$, then $T$ is said to have the {cut point property} for $G$. Set
	\[
	\mathcal{C}(G)\coloneqq \{\emptyset \}\cup \{T\subseteq [n]\mid T\text{ has the cut point property for } G\}.
	\]
	Then the minimal primes of $J_G$ correspond to the subset of the vertex set $V(G)$ with the cut point property.
	
	\begin{theorem}\cite[Corollary 3.9]{HHHFT2010-BEI-and-conditional-independence}
		Let $G$ be a graph on $[n]$ and $T\subseteq [n]$. Then $P_T(G)$ is a minimal prime ideal of $S/J_G$ if and only if $T\in \mathcal{C}(G)$.
	\end{theorem}

	The notion of \textit{decomposable graphs} were introduced in \cite{OntheextremalBettinumofBEIofblockgraphds} to study the Betti numbers of binomial edge ideals of block graphs. We briefly recall the definition here. A graph $G$ is called \emph{decomposable} if there exist subgraphs $G_1$ and $G_2$ of $G$ such that $E(G)=E(G_1)\cup E(G_2)$ with $V(G_1)\cap V(G_2)=\{v\}$, where $v$ is a free vertex in both $G_1$ and $G_2$. A graph $G$ is called \emph{indecomposable} if it is not decomposable. Since the notion of decomposability is a primary tool to construct many of the graphs in this article, we fix the following notation:
	\begin{notation}
		Following the above notation, if $G$ is decomposable, we write \[G=G_1 \cup_{v} G_2,\] 
		and say that this is a \emph{decomposition} of $G$.
	\end{notation}

    The counterpart of decomposition is the clique sum of two graphs at the vertices that are free in both the graphs. Specifically, if $v\in G_1, w\in G_2$ are free vertices, and $G$ is the clique sum of $G_1$ and $G_2$ at the vertices $v$ and $w$, then $G=G_1\cup_{v=w}G_2$. We now recall two results from \cite[Proposition 3]{OntheextremalBettinumofBEIofblockgraphds}, and \cite[Theorem 3.1]{RegBEIcertainblockgraph} which   be used in the construction of graphs.
	
	\begin{theorem}
		\label{thm: reg and mult under decomposition}
		Let $G=G_1\cup_{v} G_2$ be a decomposition of $G$. Then 
		\begin{enumerate}
			\item \label{thm: reg under decomposition} $\reg(S/J_G)=\reg(S/J_{G_1})+\reg(S/J_{G_2}).$
			\item \label{thm: HS under decomposition} $\hs({S/J_G},t)=(1-t)^2\hs({S_1/J_{G_1}},t)\hs({S_2/J_{G_2}},t),$ where $S_i, i=1,2$ are the corresponding polynomial rings associated to $J_{G_i}, i=1,2$.
		\end{enumerate}
	\end{theorem}
	
	Another useful operation is the join of two graphs. Let $H$ and $H'$ be two graphs. The \emph{join} of $H$ and $H'$, denoted by $H*H'$, is the graph on the vertex set $V(H*H')=V(H)\cup V(H')$ with edge set $E(H*H')=E(H)\cup E(H')\cup \{ij\mid i\in V(H)\text{ and }j\in V(H')\}$. The behavior of Hilbert series \cite[Theorem 4.13]{KS2019-Hilbert-series-of-BEI} and the regularity  \cite[Theorem 2.1]{MK2018-BEI-of-regularity-3} under the join of two graphs is given as follows.
	
    \begin{theorem}\label{thm: HS and reg of join of graphs}
		Let $H$ and $H'$ be two graphs on $[p]$ and $[q]$, respectively, and $G=H*H'$ be the join of $H$ and $H'$. Suppose $S_H=\K[x_i,y_i:i\in V(H)]$, $S_{H'}=\K[w_j,z_j:j\in V(H')]$ and $S=\K[x_i,y_i,w_j, z_j\mid i\in V(H),j\in V(H')]$. Then the following statements hold:
        \begin{enumerate}
            \item \label{thm: HS of join of graphs} $\hs(S/J_G,t)=\hs(S_H/J_H,t)+\hs(S_{H'}/J_{H'},t)+\frac{(p+q-1)t+1}{(1-t)^{p+q+1}}-\frac{(p-1)t+1}{(1-t)^{p+1}}-\frac{(q-1)t+1}{(1-t)^{q+1}}.$

            \item \label{thm: reg of join of graphs} If at least one of the graphs $H$ and $H'$ is not complete, then
            \[\reg(S/J_{G})= \max \{\reg(S/J_{H}),\reg(S/J_{H'}), 2\}.\]
        \end{enumerate}
	\end{theorem}

	The characterization of graphs $G$ for which $\reg(S/J_G)=2$ is known in the literature, and is given as follows.
	\begin{theorem}\cite[Theorem 3.2]{MK2018-BEI-of-regularity-3}\label{thm: BEI with reg equal to 3}
		Let $G$ be a non-complete graph with $n$ vertices and no isolated vertices. Then $\reg(S/J_G)=2$ if and only if either 
		\begin{enumerate}
			\item $G=K_r\sqcup K_s$ with $r,s\geq 2$ and $r+s=n$, or 
			\item $G=H_1*H_2$ where $H_i$ is a graph on $n_i<n$ vertices such that $n_1+n_2=n$ and $\reg(S/J_{H_i})\leq2$ for $i=1,2$.
		\end{enumerate}
	\end{theorem}

	\section{Main Result}\label{sec: main}
	This section contains the main result of the article. We begin with a lemma that is the backbone of our construction of graphs with prescribed Hilbert coefficients. In this lemma, we show that the Hilbert coefficients of the intersection of two ideals, whose heights differ in a specific manner, agree with those of the ideal of smaller height. 
	
	\begin{lemma}\label{lem: Hilbert coef. for intersection}
		Let $I,J\subseteq R=\K[x_1,\dots , x_n]$ be ideals and $i\geq 0$ an integer. If $\height(I)>i+\height(J)$, then 
        \[ 
        e_i(R/(I\cap J))=e_i(R/J).
        \]
	\end{lemma}
	\begin{proof}
		Consider the short exact sequence 
		\[
		0\longrightarrow{\frac{R}{I\cap J}}\longrightarrow \frac{R}{I}\oplus \frac{R}{J}\longrightarrow \frac{S}{I+J}\longrightarrow 0.
		\]
		By the additivity of the Hilbert series, we obtain 
		\[
		\hs(R/(I\cap J),t)=\hs(R/I,t)+\hs(R/J,t)-\hs(R/(I+J),t).
		\]
		Thus, there exist polynomials $P_{1}(t), P_2(t),Q(t)\in \mathbb Z[t]$ such that 
        \[
        \hs(R/I,t)=\frac{P_{1}(t)}{(1-t)^{n-\height(I)}}, \ \hs(R/J,t)=\frac{P_{2}(t)}{(1-t)^{n-\height(J)}}, \ \hs(R/(I+J),t)=\frac{Q(t)}{(1-t)^{n-\height(I+J)}}
        \]
        with $P_{1}(1), P_{2}(1), Q(1)\neq 0$. Since $\height(I+J)\geq \height(I)>\height(J)+i$, we can write the Hilbert series of $R/(I\cap J)$ as 
		\begin{equation}\label{eqn 1}
			\hs(R/(I\cap J),t)=\frac{(1-t)^{\height(I)-\height (J)}P_1(t)+P_2(t)-(1-t)^{\height(I+J)-\height(J)}Q(t)}{(1-t)^{n-\height(J)}}.
		\end{equation}
		Set $P(t)=(1-t)^{\height(I)-\height (J)}P_1(t)+P_2(t)-(1-t)^{\height(I+J)-\height(J)}Q(t)$. Then $P(1)=P_2(1)\neq 0$, so the \Cref{eqn 1} is a reduced Hilbert series of $R/(I\cap J)$. Hence, $e_i(S/(I\cap J))=\frac{P^{(i)}(1)}{i!} $. Note that $\height(I)-\height(J)>i$ and $\height(I+J)-\height(J)>i$, and so, the first and third terms in $P(t)$ vanish after taking $i$-th derivatives at $t=1$. Therefore, we have $P^{(i)}(1)=P_2^{(i)}(1)$, which implies 
        \[
        e_i(S/(I\cap J))=\frac{P^{(i)}(1)}{i!}=\frac{P^{(i)}_2(1)}{i!}=e_i(S/J),
        \]
        as desired.
	\end{proof}
    
    We now arrive at the main results of this paper. For the remainder of the article, let $G$ be a simple graph on the vertex set $[n]$, and let $J_G$ denote the binomial edge ideal of $G$ in the polynomial ring $S=\K[x_1, \dots, x_n, y_1, \dots, y_n]$. In the next result, we establish a combinatorial condition on a vertex that guarantees the Hilbert coefficients remain unchanged after deleting that vertex. This reduction method plays a crucial role in our construction of graphs with prescribed Hilbert coefficients. Before stating the theorem, we make the following definition.

    \begin{definition}\label{def: FCD}
        Let $v \in V(G)$, and let $C$ be a maximal clique containing $v$. We call $C$ a \emph{free clique} of $v$ if either $G = C$, or every vertex of $C$ other than $v$ is a free vertex. The \emph{free clique degree} of $v$, denoted by $\fcdeg_G(v)$, is the number of distinct free cliques of $v$.
    \end{definition}

	\begin{theorem}\label{thm: delete vertex}
		Let $G$ be a graph and $v\in V(G)$ be such that $\fcdeg_G(v)\geq i+3$ for some  $i\geq 0$. Then 
		\[
		e_i(S/J_G)=e_i(S/J_{G\setminus{v}}).
		\]
	\end{theorem}
    \begin{proof}
        Let $G$ be a graph on $[n]$ and $v\in V(G)$. Fix $i\geq 0$ and assume that $\fcdeg_G(v)=k\geq i+3$. By \cite[Lemma 4.8]{GraphsandIdealsGeneratedbysome2-minor-Ohtani-Masahiro} we have the following decomposition 
        \begin{equation}\label{eqn: decomposition of J_G}
                    J_G=J_{G_v}\bigcap \left (J_{G\setminus{v}}+ \left \langle x_v,y_v \right \rangle \right ),
        \end{equation}
        where $G_v$ is the graph obtained from $G$ by adding edges between all pairs of non-adjacent neighbors of $v$, that is,
        \[
        V(G_v)=V(G)\ \ \text{ and }\ \ E(G_v)=E(G)\cup \{ij \mid i,j\in N_G(v)\}.
        \]
        Let $W\subseteq V(G_v)$ be such that $W$ has the cut point property, and 
        \[
        \height(J_{G_v})=\height(P_W(G_v))=n+|W|-c_{G_v}(W).
        \]
        Let $H_1, \dots , H_k$ be the free cliques of $v$ in $G$ containing the vertex $v$, where each $H_j$ is a complete graph $K_{r_j}$ for some $r_j\geq 2$. Let $L$ be the induced subgraph of $G_v$ on the vertex set $\cup_{j=1}^kV(H_j)$. Then $L$ is a complete graph, and moreover, any vertex $u\in V(L)$ in the graph $G_v$ is a free vertex. Consequently, we have $V(H_j)\cap W=\emptyset$ for all $1\leq j\leq k$. Ser $r=c_{G_v}(W)$, and assume that 
        \[
        G_v\setminus W= G_1\sqcup \cdots \sqcup G_{r}.
        \]
        Then $L \subseteq G_j$ for some $1\leq j\leq r$. Without any loss of generality, assume that $j=1$. Now if $G_j'=G_j\setminus \{ab \mid a,b\in N_G(v), ab\in E(G_v)\setminus E(G)\}$, then we have
        \[
        G\setminus W= G_1' \sqcup \cdots \sqcup G_r',
        \]
        and this implies that $c_G(W)\geq r$. Set $T=W\cup \{v\}$. Note that $L' \coloneqq L\setminus \{ ab \mid a,b\in N_G(v), ab\in E(G_v)\setminus E(G)\}\subseteq G_1'$, it follows that $\fcdeg_{G_1'}(v)\geq k$. Therefore,
        \[
        G\setminus T= (H_1\setminus v)\sqcup \cdots \sqcup (H_k\setminus v)\sqcup (G_1'\setminus V(\mathcal{Z})) \sqcup G_2' \sqcup \cdots \sqcup G_r'.
        \]
        It may happen that $G_1'\setminus V(\mathcal{Z}) =\emptyset$. In any case, the above yields $c_{G}(T)\geq k+r-1$. In other words, $c_{G}(W\cup \{v\})\geq c_{G_v}(W)+k-1$. Note that $J_{G\setminus{v}}\subseteq P_{T\setminus v}(G\setminus v)$, and hence $\height\left (J_{G\setminus{v}}\right ) \leq \height\left (P_{T\setminus v}(G\setminus v)\right)$. Now, 
        \begin{alignat*}{3}
            \height\left (P_T(G)\right ) & = n+|T|-c_G(T) \qquad && \text{}\\
            & = n+|W|+1-c_G(W\cup \{v\}) \qquad && \text{[Since $T=W\cup \{v\}$]}\\
            &\leq n+|W|+1 -c_{G_v}(W)-k+1  \quad && \text{[Since $c_{G}(W\cup \{v\})\geq c_{G_v}(W)+k-1$]}\\
            & =  n+|W|-c_{G_v}(W)-k+2 \qquad && \\
            & =\height(P_W(G_v))-k+2 \qquad && \\
            & =\height(J_{G_v})-k+2. \qquad &&
        \end{alignat*}

        On the other hand, since $P_T(G)=\left\langle x_v,y_v \right \rangle+ P_{T\setminus v}(G\setminus v)$, we have $\height\left (P_T(G)\right )= 2+ \height\left (P_{T\setminus v}(G\setminus v)\right)$. Thus, the above yields $\height\left (J_{G\setminus{v}}\right ) \leq \height(J_{G_v})-k$. Now, 
        \begin{alignat*}{3}
            \height\left (J_{G_v}\right )& \geq \height\left (J_{G\setminus{v}}\right ) +k \qquad && \text{}\\
            &\geq \height\left (J_{G\setminus{v}}\right ) +i+3 \qquad && \text{[Since $k\geq i+3$]}\\
            &= \height\left (J_{G\setminus{v}}+ \left \langle x_v,y_v \right \rangle \right ) +i+1 \qquad && \text{}\\
            &> \height\left (J_{G\setminus{v}}+ \left \langle x_v,y_v \right \rangle \right ) +i. \qquad && \text{}\\
        \end{alignat*}
        Finally, applying \Cref{lem: Hilbert coef. for intersection}  together with \Cref{eqn: decomposition of J_G}, we obtain 
        \[
        e_i\left (S/J_G \right )=e_i\left (\frac{S}{{J_{G\setminus{v}}+\left\langle x_v,y_v\right \rangle}}\right )= e_i\left (\frac{S}{{J_{G\setminus{v}}}}\right ),
        \]
        and this completes the proof.
    \end{proof}

    \begin{example}\label{example}
        In \Cref{picture: free clique degree}, the vertex $w$ has free-clique degree $\fcdeg_G(w)=4$. Thus, it follows from \Cref{thm: delete vertex} that $e_0(S/J_G)=e_0(S/J_{G\setminus w})$ and $e_1(S/J_G)=e_1(S/J_{G\setminus w})$. Moreover, $\fcdeg_{G\setminus w}(v)=3$ and thus $e_0(S/J_G)=e_0(S/J_{G\setminus w})=e_0(S/J_{G\setminus \{v,w\}})$.

\begin{figure}[h!]
\centering

\begin{minipage}{0.45\textwidth}
\centering
\begin{tikzpicture}[scale=0.5, every node/.style={font=\small},
    line cap=round, line join=round, line width=0.8pt]
    \tikzstyle{vertex}=[circle, fill=black, minimum size=4pt, inner sep=0pt]

    \node[vertex, label=below:] (v1) at (0,0) {};
    \node[vertex, label=below:] (v2) at (2,0) {};
    \node[vertex, label=left:$v$] (v3) at (4,1) {}; 
    \node[vertex, label=above:] (v4) at (0,2) {};
    \node[vertex, label=above:] (v5) at (2,2) {};
    \node[vertex, label=right:$w$] (v6) at (6,1) {};

    \draw (v1)--(v2);
    \draw (v1)--(v4);
    \draw (v2)--(v3);
    \draw (v2)--(v5);
    \draw (v3)--(v6);
    \draw (v4)--(v5);
    \draw (v5)--(v3);

    \node[vertex, label=above:] (w1) at (3.5,2.3) {};
    \node[vertex, label=above:] (w2) at (5,2.3) {};
    \draw (v3)--(w1);
    \draw (v3)--(w2);
    \draw (v6)--(w2);
    \node[vertex, label=below:]  (k1) at (3.3,-0.3) {};
    \node[vertex, label=below:] (k2) at (5,-0.3) {};
    \draw (v3)--(k1);
    \draw (v3)--(k2);
    \draw (k1)--(k2);
    \draw (v6)--(k2);

    \node[vertex, label=right:] (u1) at (8,0.5) {};
    \draw (v6)--(u1);

    \node[vertex, label=right:] (u2) at (8,1.8) {};
    \draw (v6)--(u2);

    \node[vertex, label=below:]  (u3) at (7,-0.3) {};
    \node[vertex, label=below:] (u4) at (8,-0.3) {};
    \draw (v6)--(u3);
    \draw (v6)--(u4);
    \draw (u3)--(u4);

    \node[vertex, label=above:] (u5) at (6.5,3.3) {};
    \node[vertex, label=above:] (u6) at (7.7,3.7) {};
    \node[vertex, label=above:] (u7) at (8.9,3.3) {};
    \draw (v6)--(u5);
    \draw (v6)--(u6);
    \draw (v6)--(u7);
    \draw (u5)--(u6);
    \draw (u5)--(u7);
    \draw (u6)--(u7);
     \draw (v1)--(v5);
      \draw (v2)--(v4);
\end{tikzpicture}
\vspace{3pt}
\caption*{$G$}
\end{minipage}
\begin{minipage}{0.45\textwidth}
\centering
\begin{tikzpicture}[scale=0.5, every node/.style={font=\small},
    line cap=round, line join=round, line width=0.8pt]
    \tikzstyle{vertex}=[circle, fill=black, minimum size=4pt, inner sep=0pt]
    \node[vertex, label=below:] (v1) at (0,0) {};
    \node[vertex, label=below:] (v2) at (2,0) {};
    \node[vertex, label=right:$v$] (v3) at (4,1) {};
    \node[vertex, label=above:] (v4) at (0,2) {};
    \node[vertex, label=above:] (v5) at (2,2) {};
    \draw (v1)--(v2);
    \draw (v1)--(v4);
    \draw (v2)--(v3);
    \draw (v2)--(v5);
    \draw (v4)--(v5);
    \draw (v5)--(v3);
    \node[vertex, label=above:] (w1) at (3.5,2.3) {};
    \node[vertex, label=above:] (w2) at (4.5,2.3) {};
    \draw (v3)--(w1);
    \draw (v3)--(w2);
    \node[vertex, label=below:]  (k1) at (3.3,-0.3) {};
    \node[vertex, label=below:] (k2) at (4.7,-0.3) {};
    \draw (v3)--(k1);
    \draw (v3)--(k2);
    \draw (k1)--(k2);
 \draw (v1)--(v5);
  \draw (v2)--(v4);
    \node[vertex, label=right:] (u1) at (8,0.5) {};
    \node[vertex, label=right:] (u2) at (8,1.8) {};
    \node[vertex, label=below:]  (u3) at (7,-0.3) {};
    \node[vertex, label=below:] (u4) at (8,-0.3) {};
    \draw (u3)--(u4);

    \node[vertex, label=above:] (u5) at (6.5,3.3) {};
    \node[vertex, label=above:] (u6) at (7.7,3.7) {};
    \node[vertex, label=above:] (u7) at (8.9,3.3) {};
    \draw (u5)--(u6);
    \draw (u5)--(u7);
    \draw (u6)--(u7);

\end{tikzpicture}
\vspace{3pt}
\caption*{$G\setminus w$}
\end{minipage}
 \caption{}\label{picture: free clique degree}
\end{figure}
    \end{example}

    The hypothesis that $\fcdeg_G(v)\geq i+3$ in \Cref{thm: delete vertex} cannot, in general,  be weakened to the condition $\cdeg_G(v)\geq i+3$. This is illustrated in the following example. Let $J_G$ be the binomial edge   
    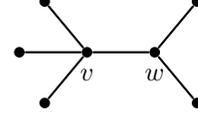
\begin{wrapfigure}{r}{0.35\textwidth}
		\centering
		\begin{tikzpicture}[scale=0.45, every node/.style={font=\small}, line cap=round, line join=round, line width=0.8pt]
        \tikzstyle{vertex}=[circle, fill=black, minimum size=4pt, inner sep=0pt]

        \node[vertex, label=below:$v$] (v) at (0,0) {};
        \node[vertex, label=below:$w$] (w) at (2,0) {};
        \node[vertex, label=below:] (a) at (-2,0) {};
        \node[vertex, label=left:]  (b) at (-1.25,1.5) {};
        \node[vertex, label=left:]  (c) at (-1.25,-1.5) {};
        \node[vertex, label=right:] (d) at (3.25,1.5) {};
        \node[vertex, label=right:] (e) at (3.25,-1.5) {};

        \draw (v)--(w);
        \draw (v)--(a);
        \draw (v)--(b);
        \draw (v)--(c);
        \draw (w)--(d);
        \draw (w)--(e);
    \end{tikzpicture}
    \caption{The graph $B_{3,2}$}\label{FigSharpBothBounds}
	\end{wrapfigure}
    ideal of the graph displayed in \Cref{FigSharpBothBounds}. In this case, it is straightforward to verify that
        \begin{flalign*}
            &\hs(S/J_G,t)=\frac{1+4t+4t^2-4t^3-5t^4+4t^5}{(1-t)^{10}},&& \\
            &\hs(S/J_{G\setminus v},t)=\frac{1+2t+t^2}{(1-t)^{10}},&&\\
            &\hs(S/J_{G\setminus w},t)=\frac{1+2t-2t^3}{(1-t)^{10}}.&&
        \end{flalign*}
    Therefore, $e_0(S/J_G)=4=e_0(S/J_{G\setminus{v}})$, but $e_0(S/J_{G\setminus{w}})=1$.

\section{Hilbert Coefficients and Regularity}\label{sec: application}

    In this section, we present an application of the main result \Cref{thm: delete vertex}. Specifically, we provide constructions of graphs $G$ corresponding to a given pair $(r,s)$, where $\reg(S/J_G)=r$ and $e_i(S/J_G)=s$ for some $i\geq 0$. Of course, the very first step is the case when $r=1$, that is, when $J_G$ has a linear resolution. This was first characterized in \cite[Theorem 2.1]{MadaniKiani2012}, which shows that $\reg(S/J_G)=1$ if and only if $G$ is a complete graph. Note that the Hilbert series of the complete graph $K_n, n\geq 2$ is given by  
	\[
	\hs(S/J_{K_n},t)=\frac{1+(n-1)t}{(1-t)^{n+1}}.
	\]
	Therefore, $e_0(S/J_{K_n})=n$, $e_1(S/J_{K_n})=n-1$, and $e_i(S/J_{K_n})=0$ for all $i\geq 2$.	Thus, in this case, all the Hilbert coefficients are completely determined and characterized.

    \par
	In the rest of the article, we focus on the cases when $\reg(S/J_G)\geq 2$. Before we move on, let us fix some notation. For $s,t\geq 1$, let $B_{s,t}$ be the graph such that
    \begin{align*}
        V(B_{s,t})&=\{u,v,u_i,v_j:1\leq i\leq s\text{ and }1\leq j\leq t\},\\
        E(B_{s,t})&=\{{uv}\}\cup \{uu_i\mid 1\leq i\leq s\}\cup\{vv_j\mid 1\leq j\leq t\}.
    \end{align*}
    We call $B_{s,t}$ the \emph{biclaw} graph with $s$ and $t$ free vertices. We   make repeated use of this graph in the constructions presented in the remainder of the article. The following two propositions address the pairs $(r,s)$ for $r=2,3$ and for any $s=e_i(S/J_G)$, where $i=0,1$.	
	\begin{proposition}\label{prop: r=2or3 e_0 e_1=s>0}
		For any integer $s\geq 1$, the following statements hold:
		\begin{enumerate}
			\item \label{prop: r=2or3 e_0 e_1=s>0 part 1} There is a graph $G$ such that $\reg(S/J_G)=2$ and $e_0(S/J_G)=s$.
			\item \label{prop: r=2or3 e_0 e_1=s>0 part 2} There is a graph $G$ such that $\reg(S/J_G)=3$ and $e_0(S/J_G)=s$.
			\item \label{prop: r=2or3 e_0 e_1=s>0 part 3} There is a graph $G$ such that $\reg(S/J_G)=2$ and $e_1(S/J_G)=s$.
			\item \label{prop: r=2or3 e_0 e_1=s>0 part 4} There is a graph $G$ such that $\reg(S/J_G)=3$ and $e_1(S/J_G)=s$.
		\end{enumerate}
	\end{proposition}
	\begin{proof}
		(1) Let $H=K_{s}\sqcup\{u, v\}, H'=\{w\}$, and set $G=H*H'$. Then $\reg(S/J_G)=2$ by \Cref{thm: BEI with reg equal to 3}. Note that $\fcdeg_G(w)=3$. Hence by  \Cref{thm: delete vertex}, we obtain $e_0(S/J_G)=e_0(S/J_{G\setminus{w}})$. The graph $G\setminus{w}$ is the disjoint union of $K_{s}$ and two isolated vertices, which implies that $e_0(S/J_{G\setminus{w}})=s$. Therefore, $e_0(S/J_{G})=s$. 

        \medskip \noindent
		(2) For the graph $B_{3,3}$, it follows from \cite[Theorem 4.1, Theorem 4.2]{RegBEIcertainblockgraph} that $\reg(S/J_{B_{3,3}})=3$ and by \Cref{thm: delete vertex}, we get $e_0(S/J_{B_{3,3}})=1$. Now assume that $s\geq 2$, and set $G=K_{s}\cup_v K_{1,4}$, where $v$ is a free vertex in both $K_s$ and $K_{1,4}$. By \Cref{thm: reg and mult under decomposition}\eqref{thm: reg under decomposition}, we get 
        \[
        \reg(S/J_{G})=\reg(S/J_{K_{s}})+\reg(S/J_{K_{1,4}})=3.
        \]
        Since the center vertex of $K_{1,4}$ in the graph $G$ has free clique degree equal to $3$, by \Cref{thm: delete vertex} we get $e_0(S/J_{G})=e_0(S/J_{K_s})=s$.
        
        \medskip \noindent
		(3) Let $H=K_{s+1}\sqcup\{u_1,u_2,u_3\}$, $H'=\{w\}$, and set $G=H*H'$. By \Cref{thm: BEI with reg equal to 3}, we get $\reg(S/J_{G})=2$. Note that $\fcdeg_G(w)=4$, and hence by \Cref{thm: delete vertex} we obtain, 
        \[
        e_1(S/J_{G})=e_1(S/J_{G\setminus{w}})=e_1(S/J_{K_{s+1}})=s.
        \]

        \medskip \noindent
		(4) By similar arguments as above, it is easy to verify that if $G=K_{s+1}\cup_{v}K_{1,5}$, where $v$ is a free vertex in both $K_{s+1}$ and $K_{1,5}$, then $\reg(S/J_{G})=3$ and $e_1(S/J_{G})=s$.
	\end{proof}

	\begin{proposition}\label{prop: r=2or3 e_1=s<=0}
		Let $s\leq 0$ be an integer. Then the following statements are true:
		\begin{enumerate}
			\item \label{prop: r=2or3 e_1=s<=0 part 1} There is a graph $G$ such that $\reg(S/J_G)=2$ and $e_1(S/J_G)=s$.
			\item \label{prop: r=2or3 e_1=s<=0 part 2} There is a graph $G$ such that $\reg(S/J_G)=3$ and $e_1(S/J_G)=s$.
		\end{enumerate}
	\end{proposition}
	\begin{proof}
		(1) Let $H=(K_{r+2}\cup_{u}K_2)\sqcup\{v\}$ be a graph on the vertex set $[r+4]$ for $r\geq 1$, where $u$ is a free vertex in both $K_{r+2}$ and $K_2$. Further, take $H'=K_2$ and set $G=H*H'$. By Theorem \ref{thm: BEI with reg equal to 3}, we have $\reg(S/J_{G})=2$. By \Cref{thm: reg and mult under decomposition}\eqref{thm: HS under decomposition}, the Hilbert series of $S/J_{H}$ is given by
		\[
		\hs(S/J_H,t)=(1-t)^2\frac{1+(r+1)t}{(1-t)^{r+3}}\frac{1+t}{(1-t)^{3}}\frac{1}{(1-t)^2}=\frac{1+(2+r)t+(r+1)t^2}{(1-t)^{r+6}}.
		\]
		Using \Cref{thm: HS and reg of join of graphs}\eqref{thm: HS of join of graphs}, we obtain
		\begin{align*}
			\hs(S/J_{G},t)=&\frac{1+(2+r)t+(r+1)t^2}{(1-t)^{r+6}}+\frac{1+(r+5)t}{(1-t)^{r+7}}-\frac{1+(r+3)t}{(1-t)^{r+5}}\\
			=&\frac{(1-t)[1+(2+r)t+(r+1)t^2]+1+(r+5)t-(1-t)^2[1+(r+3)t]}{(1-t)^{r+7}}\\
			=&\frac{1+(r+5)t+(1-t)(2r+4)t^2}{(1-t)^{r+7}}.
		\end{align*}
		Let $P(t)=1+(r+5)t+(1-t)(2r+4)t^2$. Note that the Hilbert series of $S/J_G$ is reduced, so we have $e_1(S/J_{G})=P'(1)$. A direct computation gives $P'(1)=r+5-2r-4=1-r$. As $s\leq 0$, choosing $r=1-s$ yields $e_1(S/J_{G})=s$, as required.

        \medskip \noindent
		(2) Let $G$ be the graph constructed in (1) and set $G_1=G\cup_{v}K_2$, where $v$ is a free vertex in both $G$ and $K_2$. Note that $G_1$ is a graph on the vertex set $[r+5]$. By \Cref{thm: reg and mult under decomposition}\eqref{thm: reg under decomposition}, we obtain $\reg(S/J_{G_1})=\reg(S/J_{G})+1=3$. Moreover, using \Cref{thm: reg and mult under decomposition}\eqref{thm: HS under decomposition}, the Hilbert series of $S/J_{G_1}$ is 
		\begin{align*}
			\hs(S/J_{G_1},t)&=(1-t)^2\frac{1+(r+5)t+(1-t)(2r+4)t^2}{(1-t)^{r+7}}\frac{1+t}{(1-t)^3}\\
			&=\frac{(1+t)[1+(r+5)t+(1-t)(2r+4)t^2]}{(1-t)^{r+8}}.
		\end{align*}
		Let $Q(t)=(1+t)P(t)$, where $P(t)=1+(r+5)t+(1-t)(2r+4)t^2$. Since $Q(1)\neq 0$, the above Hilbert series is reduced, and hence $e_1(S/J_{G_1})=Q'(1)$ as $Q(1)\neq 0$. Because $Q'(1)=2P'(1)+P(1)$, we get $e_1(S/J_{G_1})=8-r$. Finally, choosing $r=-s+8$ gives get $e_1(S/J_{G_1})=s$, as required.
	\end{proof}

    \par
	Next, we turn to the construction of graphs whose binomial edge ideals have prescribed Hilbert coefficients $e_i(S/J_G)$, where $i\geq 2$. As before, we begin with cases in which the regularity of the binomial edge ideal is small, specifically, when $\reg(S/J_G)\leq 3$. With this in mind, we consider the following constructions of graphs.  
	\begin{construction}\label{construction: positive even e_i and negative odd e_i}\normalfont
		Let $H=K_{m}\sqcup\{v\}$, and $H'=K_{n}$ for integers $m,n\geq 1$. By \Cref{thm: HS and reg of join of graphs}\eqref{thm: HS of join of graphs} the Hilbert series of the join $H*H'$ is
		\begin{align*}
			\hs(S/J_{H*H'},t)&=\frac{1+(m-1)t}{(1-t)^{m+3}}+\frac{1+(m+n)t}{(1-t)^{m+n+2}}-\frac{1+mt}{(1-t)^{m+2}}\\
			&=\frac{mt^2}{(1-t)^{m+3}}+\frac{1+(m+n)t}{(1-t)^{m+n+2}}\\
			&=\frac{1+(m+n)t+mt^2(1-t)^{n-1}}{(1-t)^{m+n+2}}.
		\end{align*}
		Let $P(t)=1+(m+n)t+mt^2(1-t)^{n-1}$. Note that $P(1)\neq 0$, so the above is a reduced Hilbert series. Since $\deg (P(t))=n+1$, it follows that 
        \[
        e_{n+1}(S/J_{H*H'})=\frac{P(1)^{(n+1)}}{(n+1)!}=m(-1)^{n-1}.
        \]
        Moreover, by \Cref{thm: HS and reg of join of graphs}\eqref{thm: reg of join of graphs}, it follows that $\reg(S/J_{H*H'})=2$. Therefore, for any $i\geq 1$, we conclude the following:
		\begin{enumerate}
			\item Given an integer $s\geq 1$, there is a graph $G$ such that 
            \[
            \reg(S/J_G)=2 \text{ and } e_{2i}(S/J_G)=s.
            \]
			\item Given an integer $s\leq -1$, there is a graph $G$ such that 
            \[
            \reg(S/J_G)=2 \text{ and } e_{2i+1}(S/J_G)=s.
            \]
		\end{enumerate}
	\end{construction}
    
	\begin{construction}\label{construction: negative even e_i and positive odd e_i}\normalfont
		Let $H=K_{n}\sqcup\{v_1, \dots ,v_r\}, n,r\geq 1$, $H'=\{v\}$, and $S$ be the polynomial ring corresponding to the ideal $J_{H*H'}$. Note that the Hilbert series of  $J_H$ is 
        \[
        \hs(S/J_H)=\frac{1+(n-1)t}{(1-t)^{n+1+2r}}.
        \]
        Applying \Cref{thm: HS and reg of join of graphs}\eqref{thm: HS of join of graphs}, we obtain 
		\begin{align*}   
			\hs(S/J_{H*H'},t)&=\frac{1+(n-1)t}{(1-t)^{n+1+2r}}+\frac{1+(n+r)t}{(1-t)^{n+2+r}}-\frac{1+(n+r-1)t}{(1-t)^{n+1+r}}\\
			&= \frac{1+(n-1)t+(1-t)^{r-1}(1+(n+r)t)-(1-t)^r(1+(n+r-1)t)}{(1-t)^{n+1+2r}}\\
			&= \frac{1+(n-1)t+(1-t)^{r-1}[1+(n+r)t-(1-t)(1+(n+r-1)t)]}{(1-t)^{n+1+2r}}\\
			&= \frac{1+(n-1)t+(1-t)^{r-1}[2t+(n+r-1)t^2]}{(1-t)^{n+1+2r}}.
		\end{align*}
		Next, consider the graph $G=(K_{n}\sqcup\{u,v,w\})*K_m, n\geq 1, m\geq 3$. Again using \Cref{thm: HS and reg of join of graphs}\eqref{thm: HS of join of graphs} and the above computation, we get
		\begin{align*}   
			\hs(S/J_G,t)=& \frac{1+(n-1)t+(1-t)^2(2t+(n+2)t^2)}{(1-t)^{n+7}}+\frac{1+(n+m+2)t}{(1-t)^{n+m+4}}-\frac{1+(n+2)t}{(1-t)^{n+4}},
		\end{align*}
		and after simplifications, we get 
		\[
		\hs(S/J_G,t)= \frac{(1-t)^{m-3}[1+(n-1)t]+(1-t)^{m-1}Q(t)+1+(n+m+2)t}{(1-t)^{n+m+4}},
		\]
		where $Q(t)=-1-(n-1)t+2(n+2)t^2$. Setting 
        \[
        P(t)=(1-t)^{m-3}[1+(n-1)t]+(1-t)^{m-1}Q(t)+1+(n+m+2)t,
        \]
        we observe that $P(1)\neq 0$. Hence, the above Hilbert series is reduced, and for any $s\geq 0$, the Hilbert coefficients are given by $e_{s}(S/J_{G})=\frac{P^{(s)}(1)}{s!}$. Note that 
        \[
        P^{(m-2)}(1)=(-1)^{m-3}(m-3)!(m-2)(n-1),
        \]
        so we get $e_{m-2}(S/J_{G})=(-1)^{m-1}(n-1)$ for all $m\geq 4$. Moreover, by \Cref{thm: HS and reg of join of graphs}\eqref{thm: reg of join of graphs}, it follows that $\reg(S/J_{G})=2$. Thus for any $i\geq 1$, we conclude the following:
		\begin{enumerate}
			\item Given an integer $s\leq 0$, there is a graph $G$ such that 
            \[
            \reg(S/J_G)=2 \text{ and } e_{2i}(S/J_G)=s.
            \]
			\item Given an integer $s\geq 0$, there is a graph $G$ such that
            \[
            \reg(S/J_G)=2 \text{ and } e_{2i+1}(S/J_G)=s.
            \]
		\end{enumerate}
	\end{construction}
	
    The two constructions above provide graphs corresponding to the pairs $(2, s)$ where $s=e_i(S/J_G)$ for $i \ge 2$. The next proposition deals with the remaining cases of $r=3$, namely the pairs $(3, s)$ with $s=e_i(S/J_G)$ and $i \ge 2$. But before this, we require the following lemma.

    \begin{lemma}\label{proof of reg 3 of required graphs}
        Given integers $r\geq 2,s,t\geq 1$ and $n_i\geq 1, 1\leq i\leq r$. Assume that $n_1\geq 2$ and $w\in V(K_{n_1})$. Consider the graph $G_{r,s,t}$  where 
        \begin{align*}
            V(G_{r,s,t})&=V((K_{n_1}\sqcup K_{n_2}\sqcup\cdots\sqcup K_{n_r})*K_s)\cup\{u_1, \dots, u_t\}, \\
        E(G_{r,s,t})&=E((K_{n_1}\sqcup K_{n_2}\sqcup\cdots\sqcup K_{n_r})*K_s)\cup\{wu_1, \dots, wu_t\}. 
        \end{align*}
    Further assume that $n_i\geq 2$ for at most two $i$. Then $\reg(S/J_{G_{r,s,t}})=3$ for all $r\geq 2$ and $s,t\geq 1$. 
    \end{lemma}
    \begin{proof}
        We proceed by induction on $t$. If $t=1$ then $w$ is a free vertex in both  $\{wu_1\}$ and $(K_{n_1}\sqcup K_{n_2}\sqcup\cdots\sqcup K_{n_r})*K_s$, and hence the statement follows from \Cref{thm: reg and mult under decomposition}\eqref{thm: reg under decomposition} and \Cref{thm: BEI with reg equal to 3}.  Assume $t\geq 2$. Then by \cite[Corollary 4.4]{Partial-Betti-splittings-applications-to-BEI-Jayanthan-Shivakumar-Van-tuyl}, we get 
        \[
            \reg(S/J_{G_{r,s,t}})=\max\left\{\reg(S/J_{G_{r,s,t-1}}),1+\reg(S/J_{G'})\right\},
        \]
        where $G'=({K_{n_1+t-1}}\sqcup K_{n_2}\sqcup\cdots\sqcup K_{n_r})*K_s$. By \Cref{thm: BEI with reg equal to 3}, $\reg(S/J_{G'})=2$ and the induction hypothesis gives $\reg(S/J_{G_{r,s,t-1}})=3$ as $t\geq2$. Therefore, we get the desired result. 
    \end{proof}
	\begin{proposition}\label{prop: reg 3 and all e_i}
        Let $i\geq 2$ be an integer. Given $s\in \mathbb{Z}$, there is a graph $G$ such that 
        \[
        \reg(S/J_{G})=3 \text{ and } e_i(S/J_{G})=s.
        \]
	\end{proposition}
	\begin{proof}
		If $s=0$, consider $G=B_{i+3,i+3}$. Then it follows from \cite[Theorem 4.1, Theorem 4.2]{RegBEIcertainblockgraph} that $\reg(S/J_{G})=3$. Moreover, by Theorem \ref{thm: delete vertex}, $e_i(S/J_{G})=0$ for all $i\geq 2$. Now, assume that $s\neq 0$. Take the graph $(K_r\sqcup \{v\})*K_n$ for $r,n\geq 1$ as considered in \Cref{construction: positive even e_i and negative odd e_i}, and ~set 
        \[
        G_{r,n}=((K_r\sqcup \{v\})*K_n)\cup_v K_2,
        \]
        Therefore, by \Cref{thm: reg and mult under decomposition}\eqref{thm: reg under decomposition}, we get $\reg(S/J_{G_{r,n}})=3$ for all $r,n\geq 1$. Again by \Cref{thm: reg and mult under decomposition}\eqref{thm: HS under decomposition}, the Hilbert series of $S/J_{G_{r,n}}$ is given by 
		\[
		\hs(S/J_{G_{r,n}},t)=\frac{(1+t)[1+(n+r)t+rt^2(1-t)^{n-1}]}{(1-t)^{n+r+3}}.
		\]
		Then we have $e_{n+2}(S/J_{G_{r,n}})=r(-1)^{n+1}$ for all $r,n\geq 1$. Therefore, whenever $i\geq 4$ is even, and $s\leq -1$, we have $e_i(S/J_{G_{(-s),i-2}})=s$; and when $i\geq 3$ is odd, and $s\geq 1$, we have $e_i(S/J_{G_{s,i-2}})=s$.
		
		
		Next, we take $H=K_2$, where $V(H)=\{v,w\}$. Let us consider the graph $G'_{r,n}$ where 
        \begin{align*}
            V(G'_{r,n})&=V((K_r\sqcup H)*K_n)\cup\{u_1,\dots ,u_{n+4}\},\\
            E(G'_{r,n})&= E((K_r\sqcup H)*K_n)\cup \{wu_1, wu_2,\dots ,wu_{n+4}\}
        \end{align*}
        for $r,n\geq 1$. Then observe that $\fcdeg_{G'_{r,n}}(w)=n+4$, and hence by \Cref{thm: delete vertex}, we get 
        \[
        e_{n+1}(S/J_{G'_{r,n}})=e_{n+1}(S/J_{G'_{r,n}\setminus w}).
        \]
        Note that $G'_{r,n}\setminus w= ((K_r\sqcup\{v\})*K_n)\cup \{u_1,\dots ,u_{n+4}\}$, and hence, using \Cref{construction: positive even e_i and negative odd e_i}, we get $e_{n+1}(S/J_{G'_{r,n}})=r(-1)^{n+1}$ for all $r,n\geq 1$. Then it follows from \Cref{proof of reg 3 of required graphs} that $\reg(S/J_{G'_{r,n}})=3$ for all $r,n\geq 1$. Therefore, whenever $i\geq 2$ is even, and $s_i\geq 1$, we have $e_i(S/J_{G_{s,i-1}})=s$; and when $i\geq 3$ is odd, and $s\leq -1$, we have $e_i(S/J_{G_{(-s),i-1}})=s$.
		
		Finally, it remains the case when $i=2$ and $s\leq -1$. For this, we take \[
        H=(K_{n-1}\sqcup\{u\}\sqcup\{vw\})*K_2,
        \]
        where $n\geq 2$ and consider the graph $G$ where 
        \begin{align*}
            V(G)=V(H)\cup \{u_1,\dots ,u_5\},\ 
            E(G)=E(H)\cup\{wu_1,\dots ,wu_5\}.
        \end{align*}
        Since $\fcdeg_G(w)=5$, we get by \Cref{thm: delete vertex} that $e_2(S/J_G)=e_2(S/J_{G\setminus w})$. Note that 
        \[
        G\setminus w=((K_{n-1}\sqcup\{u\}\sqcup\{v\})*K_2)\sqcup \{u_1,\dots ,u_5\}.
        \]
        Then by \Cref{construction: negative even e_i and positive odd e_i}, we get $e_2(S/J_G)=-(n-2)$. By \Cref{proof of reg 3 of required graphs}, we get $\reg(S/J_G)=3$, and this completes the proof.
	\end{proof}

    Up to this point, we have constructed graphs realizing all possible pairs $(r, s)$ for $r \le 3$. We now turn to showing that, starting from this class of graphs, one can construct graphs with arbitrary regularity and prescribed Hilbert coefficients.
	\begin{theorem}\label{thm: reg=2 is sufficient for higher reg}
		Let $G$ be a graph such that 
        $\reg(S/J_{G})=2$ and $G$ has a free vertex. Given integers $r\geq 4$ and $i\geq 0$, there is a graph $H$ such that 
        \[
        \reg(S/J_H)=r \text{ and } e_i(S/H)=e_i(S/G).
        \]
	\end{theorem}
	\begin{proof}
        First, let us assume that $r=2m$ for some $m\geq 2$. Consider the graph
        \[
        H=G\cup_{v}K_{1,t_1}\cup_{v_1}K_{1,t_2}\cup_{v_2}\cdots \cup_{v_{m-2}}K_{1,t_{m-1}},
        \]
        where $t_j\geq i+5$ for all $1\leq j\leq m-1$, and $v$ is free vertex in $G$ and $K_{1,t_1}$, and for $1\leq j\leq m-2$, the vertex $v_j$ is free in both $K_{1,t_j}$ and $K_{1,t_{j+1}}$ with $v_j
        \neq v_k$ for $j\neq k$. Then, by \Cref{thm: reg and mult under decomposition}\eqref{thm: reg under decomposition}, we get $\reg(S/J_H)=\reg(S/J_{G})+2(m-1)=2m$. Since $t_j\geq i+5$, each of the center vertices of the graph $K_{1,t_j}$ has free clique degree $\geq i+3$. Therefore, by repeated application of \Cref{thm: delete vertex}, we obtain $e_i(S/J_H)=e_i({S/J_{H\setminus{\{c_1,\ldots, c_{m-1}\}}}})$. Notice that the graph $H\setminus{\{c_1,\ldots, c_{m-1}\}}$ is just $G$ together with some isolated vertices, and isolated vertices do not affect the Hilbert coefficients. Hence, $e_i(S/J_{H})=e_i({S/J_{G}})$.   
		
		Next, we assume that $r=2m+1$ with $m\geq 2$. Write $r=2+3+2(m-2)$, and consider
        \[H=G\cup _{u_1}B_{s,t}\cup_{w_1}K_{1,t_1}\cup_{w_2}\cdots\cup_{w_{m-1}}K_{1,t_{m-2}},
        \]
        where $u_1$ is a free vertex in both $G_1$ and $B_{s,t}$, $w_1$ is free in both $B_{s,t}$ and $K_{1,t_1}$, and for $2\leq j\leq m-2$, $w_j$ is free in both $K_{1,t_{j-1}}$ and $K_{1,t_j}$. Further, assume that $s,t,w_j\geq i+5$ for $1\leq j\leq m-2$. Then by \Cref{thm: reg and mult under decomposition} \eqref{thm: reg under decomposition} we obtain 
        \[
        \reg(S/J_H)=\reg(S/J_{G})+\reg(S/J_{B_{s,t}})+2(m-2).
        \]
        Since $\reg(S/J_{B_{s,t}})=3$, we have $\reg(S/J_G)=2+3+2(m-2)=r$. Notice that the vertices $u,v\in V(B_{s,t})$ and center vertices of $K_{1,t_j}$ have free clique degree $\geq i+3$. Now applying Theorem \ref{thm: delete vertex} repeatedly, we get 
        \[
        e_i(S/J_H)=e_i(S/J_{H\setminus{\{u,v,c_1,\ldots,c_{m-2}\}}}).
        \]
        Again $H\setminus{\{u,v,c_1,\ldots,c_{m-2}\}}$ is just $G$ together with some isolated vertices, so we have $e_i(S/J_H)=e_i(S/J_{G})$. This completes the proof.
	\end{proof}

    \begin{remark}
        In the graph constructed in \Cref{prop: r=2or3 e_0 e_1=s>0}\eqref{prop: r=2or3 e_0 e_1=s>0 part 1}, the vertices $u$ and $v$ are free. In \Cref{prop: r=2or3 e_0 e_1=s>0}\eqref{prop: r=2or3 e_0 e_1=s>0 part 3}, the vertices $u_1$, $u_2$, and $u_3$ are free, and in \Cref{prop: r=2or3 e_1=s<=0}\eqref{prop: r=2or3 e_1=s<=0 part 1} the vertex $v$ is free as well. Consequently, in the theorem above, the assumption that the graph $G$ contains a free vertex may be omitted.
    \end{remark}

    In conclusion, we arrive at the final goal of this section, which summarizes the outcome of the preceding constructions and results.
	\begin{theorem}\label{thm: main}
		Given $i\geq 0$ and a pair $(r,s)$ with $r\geq 2, s\in \mathbb Z$, there is a graph $G$ such that 
        \[
        \reg(S/J_G)=r \text{ and } e_i(S/J_G)=s.
        \]
	\end{theorem}

    \section*{Acknowledgements}
    The authors acknowledge the use of the computer algebra system Macaulay2 \cite{M2} and the online platform SageMath for testing their computations. This project began during the first author’s visit to the second and third authors at IIT Jammu, whose hospitality is gratefully acknowledged. The first author acknowledges support from the Infosys Foundation. The second author is supported by the Indian Institute of Technology Jammu, under the Seed Grant (SGR-100014).
	
	\bibliographystyle{abbrv}
	\bibliography{ref}
    \nocite{Book-Binomial-Ideals}
\end{document}